\documentclass[12pt]{amsart}
\usepackage{amsmath,amsthm,amssymb,fullpage,enumerate}

\newcommand{\ceil}[1]{\left \lceil #1 \right \rceil}
\newcommand{\floor}[1]{\left \lfloor #1 \right \rfloor}
\newcommand{\cart}{\, \Box \,}
\newcommand{\capt}{\mathrm{capt}}
\newcommand{\prob}[1]{\mathrm{Pr}\left [ #1 \right ]}
\newcommand{\expect}[1]{\mathbb{E} \left [ #1 \right ]}
\newcommand{\eps}{\varepsilon}
\newcommand{\given}{\, \vert \,}

\newcommand{\Z}{{\mathbb Z}}

\newcommand{\seq}{\mathrm{start}}

\newcommand{\dist}{\mathrm{dist}}

\newtheorem{theorem}{Theorem}
\newtheorem{lemma}[theorem]{Lemma}

\newtheorem{cor}[theorem]{Corollary}

\begin{document}
\title{The capture time of the hypercube}

\author{Anthony Bonato}
\address{Department of Mathematics, Ryerson University, Toronto, ON, Canada, M5B 2K3}
\email{\tt abonato@ryerson.ca}

\author{Przemys{\l}aw Gordinowicz}
\address{Institute of Mathematics, Technical University of Lodz, {\L}\'{o}d\'{z}, Poland}
\email{\tt pgordin@p.lodz.pl}

\author{William B. Kinnersley}
\address{Department of Mathematics, Ryerson University, Toronto, ON, Canada, M5B 2K3}
\email{\tt wkinners@ryerson.ca}

\author{Pawe{\l} Pra{\l}at}
\address{Department of Mathematics, Ryerson University, Toronto, ON, Canada, M5B 2K3}
\email{\tt pralat@ryerson.ca}

\thanks{Supported by grants from NSERC and Ryerson}

\begin{abstract}
In the game of Cops and Robbers, the capture time of a graph is the minimum number of moves needed by the cops to
capture the robber, assuming optimal play. We prove that the capture time of the $n$-dimensional hypercube is $\Theta (n\ln n).$ Our methods
include a novel randomized strategy for the players, which involves the analysis of the coupon-collector problem.
\end{abstract}

\maketitle

\section{Introduction}

The game of Cops and Robbers (defined, along with all the standard notation, at the end of this section) is usually studied in the context of the minimum number of cops needed to have a winning strategy, or \emph{cop number}. The cop number is often a challenging graph parameter to analyze, and establishing upper bounds for this parameter is the focus of Meyniel's conjecture: the cop number of a connected $n$-vertex graph is $O(\sqrt{n}).$ For additional background on Cops and Robbers and Meyniel's conjecture, see the book~\cite{bonato}.

Our focus in the present article is not on the number of cops needed, but rather, how long it takes them to win. To be more precise, the
\emph{length} of a game is the number of rounds it takes (not including the
initial or $0$th round) to capture the robber (the degenerate case is the game played on $K_1$, which has length $0$). We say that a play of the game with $c(G)$
cops is \emph{optimal} if its length is the minimum over all possible strategies for the cops, assuming the robber is trying to evade capture for as
long as possible (here $c(G)$ denotes the cop number of $G$). There may be many optimal plays possible (for example, on the path $
P_{4}$ with four vertices, the cop may start on either of the two vertices in the centre), but the length of
an optimal game is an invariant of $G.$ When $k$ cops play on a graph with $k \ge c(G),$ we denote this invariant $\mathrm{capt}_k(G),$
which we call the $k$-\emph{capture time} of $G.$ In the case $k=c(G)$, we just write
$\mathrm{capt}(G)$ and refer to this as the \emph{capture time} of $G.$  Note that $\mathrm{capt}_k(G)$ is monotonically non-increasing with $k$. The capture time parameters may be viewed as temporal counterparts to the cop number, and were introduced in~\cite{boncapt}.

To date, little is known for exact values or even bounds on the capture times of various graphs. In \cite{boncapt}, the authors proved that if $G$ is cop-win (that is, has cop number $1$) of order $n\ge5,$ then $\mathrm{capt}(G) \leq n-3.$ By considering small-order cop-win graphs, the bound was improved to $\mathrm{capt}(G)\le n-4$ for $n \ge 7$ in \cite{gav1}. Examples were given of planar cop-win graphs in both \cite{boncapt,gav1} which prove that the bound of $n-4$ is optimal. Beyond the case $k=1$, a recent paper of Mehrabian~\cite{meh} investigates the $2$-capture time of grids (that is, Cartesian products of paths). It was shown in \cite{meh} that if $G$ is the Cartesian product of two trees, then $\mathrm{capt}(G)=\lfloor \mathrm{diam}(G)/2 \rfloor ;$  hence, the $2$-capture time of an $m \times n$ grid is $\lfloor \frac{m+n}{2} \rfloor -1 .$ Other recent works on capture time where a visible or an invisible robber moves via a random walk are~\cite{kmp,kp}.

The question of determining the capture time of higher dimensional grids was left open in~\cite{meh}. Perhaps the simplest such graphs are the hypercubes, written $Q_n$, which may be viewed as the $n$-fold Cartesian products of $K_2$. In the present paper we derive the asymptotic order of the capture time of the hypercube. It was established in \cite{mm} that the cop number of a Cartesian product of $n$ trees is $\lceil \frac{n+1}{2} \rceil$; in particular, $c(Q_n) = \lceil \frac{n+1}{2} \rceil$. We note that the cop numbers of Cartesian and other graph products were investigated first in \cite{nn}.

Our main result is the following theorem, whose proof is deferred to the next sections.

\begin{theorem}
\label{time} For $n$ an integer, we have that
$$\mathrm{capt}(Q_n) = \Theta(n \ln n).$$
\end{theorem}

We consider separately the upper and lower bounds on $\mathrm{capt}(Q_n)$; see Sections~\ref{up} and \ref{lower}, respectively. The techniques used for the lower bound are especially novel, as we assume the robber moves randomly (essentially, as a random walk in the graph). Although a priori this may be counterintuitive (an optimal strategy would be for the robber to move away from the cops, which fails to occur with a given probability in a random strategy), we show that with positive probability the robber has a strategy to survive long enough to achieve the lower bound; the analysis involves consideration of the coupon-collector problem and a tail inequality on the expected amount of time needed to collect the coupons.

\subsection{Definitions and notation}
We consider only finite, reflexive, undirected graphs in the paper. For background on graph theory, the reader is directed to~\cite{west}.

The game of \emph{Cops and Robbers} was independently introduced in~\cite{nw,q} and the cop number was introduced in~\cite{af}. The game is played on a reflexive
graph; that is, vertices each have at least one loop. Multiple edges are
allowed, but make no difference to the play of the game, so we always assume there
is exactly one edge between adjacent vertices. There are two players
consisting of a set of \emph{cops} and a single \emph{robber}. The game is
played over a sequence of discrete time-steps or \emph{turns},
with the cops going first on turn $0$ and then playing on alternate time-steps.  A \emph{round} of the game is a cop move together with the subsequent robber move.  The cops and robber occupy vertices;
for simplicity, we often identify the player with the vertex they occupy. We
refer to the set of cops as $C$ and the robber as $R.$ When a player is ready to move in a round they must
move to a neighbouring vertex. Because of the loops, players can \emph{pass}, or remain on their own vertices. Observe that any subset of $C$ may
move in a given round. The cops win if after some finite number of rounds, one of them can occupy
the same vertex as the robber (in a reflexive graph, this is equivalent to the cop landing on the robber).
This is called a \emph{capture}. The robber
wins if he can
evade capture indefinitely. A \emph{winning strategy for the cops} is a set
of rules that if followed, result in a win for the cops. A \emph{winning
strategy for the robber} is defined analogously.

If we place a cop at each vertex, then the cops are guaranteed to win.
Therefore, the minimum number of cops required to win in a graph $G$ is a
well-defined positive integer, named the \emph{cop number} (or \emph{%
copnumber}) of the graph $G.$ We write $c(G)$ for the cop number of a graph $%
G$.

We use the notation $\lg n$ for the logarithm of $n$ in base 2, and $\ln n$ for the logarithm in the natural base.

\section{Upper bound}\label{up}
We establish an upper bound on the capture time of the hypercube by using a strategy for the cops which is analogous to the one described in~\cite{mm} (although the reader need not be familiar with that strategy to understand ours).  We obtain our result as a special case of a more general upper bound.

We begin with a lemma, which concerns a notion closely related to the capture time. The \emph{eccentricity} of a vertex $u$ in a connected graph $G$ is the maximum distance between $u$ and another vertex of $G$. The \emph{radius} of $G$, written $\mathrm{rad}(G)$, is the minimum eccentricity of a vertex of $G$. A \emph{central vertex} in a graph of radius $r$ is one whose eccentricity is $r.$

\begin{lemma}\label{lem:prodone}
Let $T_1$ and $T_2$ be trees, and consider the game played in $T_1 \cart T_2$ with a single cop.  If the robber moves at least $\mathrm{rad}(T_1 \cart T_2)$ times in the course of the game, then the cop can capture him (and can do so immediately after the robber's last move).
\end{lemma}
\begin{proof}
We use the cop strategy suggested in~\cite{mm}.  The cop starts on a central vertex in $T_1 \cart T_2$.  On the cop's turn, let $(u,v)$ denote the cop's position and $(u',v')$ the robber's.  Let $d_1 = \dist_{T_1}(u,u')$, let $d_2 = \dist_{T_2}(v,v')$, and let $d = d_1 + d_2$.  If $d$ is even, then the cop stands still. Otherwise, if $d_1 > d_2$, then the cop approaches the robber by changing the first coordinate of his position; if $d_2 > d_1$, then the cop approaches by changing his second coordinate.

By the cop's strategy, he might move on his first turn, but afterward moves only after the robber moves.  By symmetry, we may suppose $d_1 \ge d_2$ on the cop's first turn.  So long as $d_1 > d_2$, the cop will change only the first coordinate of his position until, after some cop move, $d_1 = d_2$.  After this point in the game, the cop's strategy ensures that $d_1 = d_2$ after each of the cop's turns, and hence, if the robber ever agrees with the cop in either coordinate, he will be captured on the cop's subsequent turn.  Consequently, when we restrict attention to the first coordinates of the cop's positions (that is, when we consider the projection onto $T_1$), the cop follows a path in $T_1$.  Likewise, the cop follows a path in the projection onto $T_2$.  Thus, the cop follows a shortest path in $T_1 \cart T_2$ to the vertex at which he eventually captures the robber.  Hence, the cop moves no more than $\mathrm{rad}(T_1 \cart T_2)$ times before capturing the robber.
\end{proof}

\begin{theorem} \label{thm:upper}
For every integer $n$, and all trees $T_0, T_1 \ldots, T_{n-1}$, we have that $$\capt(T_0 \cart T_1 \cart \cdots \cart T_{n-1}) \le \left ( \sum_{i=0}^{n-1} \mathrm{rad}(T_i) \right )\ceil{\lg n} - \floor{\frac{n-1}{2}} + 1.$$
\end{theorem}

Before proving Theorem~\ref{thm:upper}, we introduce some notation. We view the vertices of the product $T_0 \cart T_1 \cart \cdots \cart T_{n-1}$ as $n$-tuples with coordinates indexed by $0, 1, \ldots, n-1$.  Now consider the binary representations of the coordinate indices, each using $\ceil{\lg n}$ bits (where we allow leading zeros).  Let $s$ be a binary sequence of length at most $\ceil{\lg n}$.  We denote by $\seq(s)$ the set of coordinate indices whose binary representations start with $s;$ that is,
$$\seq(s) = \{i \in \Z \colon 0 \le i \le n-1 \textrm{ and } s \textrm{~is a prefix of the } \ceil{\lg n}\textrm{-bit binary representation of~} i\}.$$
Let $\circ$ denote the concatenation operator on binary strings.

\begin{proof}[Proof of Theorem~\ref{thm:upper}]
When $n \le 2$ the desired bound was proved in~\cite{meh}, so suppose $n \ge 3$.  Let $k = c(T_0 \cart T_1 \cart \cdots \cart T_{n-1}) = \lceil \frac{n+1}{2} \rceil$, and label the $k$ cops $C_0, C_1, \ldots, C_{k-1}$.  Each cop's strategy has two steps.

\vspace{0.1in}
\noindent {\bf Step 1}: The cop attempts to occupy a vertex that disagrees with the robber's vertex in only one or two coordinates.
\vspace{0.1in}

More precisely, for $0 \le i \le k-2$, cop $C_i$ seeks to agree with the robber in all coordinates except $2i$ and $2i+1$.  If $n$ is odd, then $C_{k-1}$ wants to agree with the robber in all coordinates except $n-1$; if instead $n$ is even, then $C_{k-1}$ wants to agree in all coordinates except $n-2$ and $n-1$, just like cop $C_{k-2}$.  Since the specific details of this step are more complex, we postpone them until after we discuss Step 2.

\vspace{0.1in}
\noindent {\bf Step 2}: The cop tries to actually capture the robber.
\vspace{0.1in}

Suppose first that $n$ is odd.  For $1 \le i \le k-2$, cop $C_i$ initially agrees with the robber in all coordinates except $2i$ and $2i+1$ (this was the goal of Step 1).  If the robber stands still, then so does $C_i$.  Suppose instead that the robber changes coordinate $j$ of his position.  If $j \not \in \{2i,2i+1\}$, then $C_i$ also changes coordinate $j$, in order to maintain his agreement with the robber.  Finally, if $j \in \{2i,2i+1\}$, then $C_i$ also changes one of these coordinates, in accordance with the strategy on $T_{2i} \cart T_{2i+1}$ described in Lemma~\ref{lem:prodone}.  (That is, $C_i$ considers the projections of both players' positions onto coordinates $2i$ and $2i+1$, and plays as he would on this subgraph.)  Cop $C_{k-1}$ plays as described above, except that when the robber stands still, he plays in accordance with an optimal strategy on $T_{n-1}$ (just as he does when the robber changes coordinate $n-1$).

When $n$ is even, cop $C_i$ plays just as above for $0 \le i \le k-3$.  Cops $C_{k-2}$ and $C_{k-1}$ play similarly, but with some small changes.  If the robber changes coordinate $j$ for $j \not \in \{n-2,n-1\}$, then $C_{k-2}$ and $C_{k-1}$ both change coordinate $j$ to maintain their agreement with the robber.  Otherwise, (including when the robber stands still) $C_{k-2}$ and $C_{k-1}$ together play in accordance with an optimal strategy on $T_{n-2} \cart T_{n-1}$.

\vspace{0.1in}

Now, for the completion of Step 1, the cops move as follows.  For $0 \le i \le n-1$, let $v_i$ be a central vertex in $T_i$; initially all $k$ cops start at vertex $(v_1, v_2, \ldots, v_n)$.   Each cop approaches the robber in $\ceil{\lg n-1}$ ``phases''.  During each phase, each cop deems certain coordinates \emph{active}; active coordinates remain active throughout all subsequent phases.  A coordinate that is not active is \emph{inactive}.  The cop aims to occupy a vertex that agrees with the robber's vertex in every active coordinate; once the cop has done so, he proceeds to the next phase (and chooses additional active coordinates).  On the robber's turn, he either stands still or changes exactly one coordinate of his position.  If he changes an active coordinate in which he and the cop previously agreed, then the cop changes that same coordinate; in all other cases, the cop approaches the robber in the direction of the first active coordinate in which he and the robber disagree.  (Since each $T_i$ is a tree, the cop's move is uniquely determined.)  Thus, once the cop and robber agree in an active coordinate, the cop maintains this agreement for the remainder of the game.  

We next specify how the cops choose active coordinates in Step 1. Fix a cop $C_i$ with $0 \le i \le k-2$ (we return to cop $C_{k-1}$ later). Recall that $C_i$ seeks to eventually agree with the robber in all coordinates other than $2i$ and $2i+1$; consequently, he will never deem these coordinates active.  During phase $j$, for $1 \le j \le \ceil{\lg n}-1$, a coordinate is active for $C_i$ precisely when it differs from $2i$ in one of the first $j$ bits.  (In some cases, a cop has the same active coordinates in phase $k$ as in phase $k-1$; in such cases, the cop finishes phase $k$ immediately, and proceeds directly to phase $k+1$.)  For example, when $n=11$, there are three phases and six cops; all cops' active coordinates in each phase are presented in Table~\ref{tab:example}.
Note that in phase $\ceil{\lg n} - 1$, cop $C_i$'s only remaining inactive coordinates are $2i$ and $2i+1$.
When $n$ is odd, cop $C_{k-1}$ uses the same approach; in phase $\ceil{\lg n} - 1$, only coordinate $n-1$ remains inactive.  When $n$ is even, $C_{k-1}$ moves identically to $C_{k-2}$ throughout Step 1.
\begin{table}[htb]
\centering
\begin{tabular}{|l||c|c|c|}
\hline
Cop & Phase 1 & Phase 2 & Phase 3 \\
\hline
\hline
$C_0$ &  $\{8,9,10\}$ &  $\{4,5,6,7,8,9,10\}$ &  $\{2,3,4,5,6,7,8,9,10\}$\\
$C_1$ &  $\{8,9,10\}$ &  $\{4,5,6,7,8,9,10\}$ &  $\{0,1,4,5,6,7,8,9,10\}$\\
$C_2$ &  $\{8,9,10\}$ &  $\{0,1,2,3,8,9,10\}$ &  $\{0,1,2,3,6,7,8,9,10\}$\\
$C_3$ &  $\{8,9,10\}$ &  $\{0,1,2,3,8,9,10\}$ &  $\{0,1,2,3,4,5,8,9,10\}$\\
$C_4$ &  $\{0,1,2,3,4,5,6,7\}$ &  $\{0,1,2,3,4,5,6,7\}$ &  $\{0,1,2,3,4,5,6,7,10\}$\\
$C_5$ &  $\{0,1,2,3,4,5,6,7\}$ &  $\{0,1,2,3,4,5,6,7\}$ &  $\{0,1,2,3,4,5,6,7,8,9\}$\\
\hline
\end{tabular}
\caption{Active coordinates during play on $Q_{11}.$} \label{tab:example}
\end{table}

To bound the number of rounds needed to complete Step 1, we group the cops into \emph{squads}.  A squad is a collection of cops having the same inactive coordinates. For example, from Table~\ref{tab:example} in Phase~1 we have one squad formed by $C_0,$ $C_1,$ $C_2,$ and $C_3$, and the second squad formed by $C_4$ and $C_5.$ For $1 \le i \le k$ and $1 \le j \le \ceil{\lg n} - 1$, coordinate $x$ is inactive for cop $C_i$ during phase $j$ precisely when $x$ agrees with $i$ in the first $j$ bits.  Thus, we may identify each possible squad by a binary sequence of length at most $\ceil{\lg n} - 1$; at any given point in the game, for a binary sequence $s$, let $S_s$ denote the squad of cops having inactive coordinates $\seq(s)$. 

Initially we have two nonempty squads, $S_0$ and $S_1$, and all cops begin in phase 1.  Since all cops begin at the same vertex, and cops in the same squad move identically, all members of a squad complete their phases simultaneously.  When the cops in $S_0$ finish phase 1, they all leave $S_0$ and join either $S_{00}$ or $S_{01}$.  Likewise, when $S_{01}$ finishes phase 2, it splits into $S_{010}$ and $S_{011}$.  In general, for $s$ a binary string with length $j$, such that $\seq(s) \neq \emptyset,$ squad $S_s$ is nonempty only while its members are in phase $j$.  Once the squad completes phase $j$ (for $j < \ceil{\lg n} - 1$), it splits into squads $S_{s \circ 0}$ and $S_{s \circ 1}$.  (Note that one of $S_{s \circ 0}$ or $S_{s \circ 1}$ could remain empty.  For example, on $Q_{11}$, squad $S_{11}$ is empty throughout the game; this is due to the fact that $\seq(11) = \emptyset$.)

Fix a binary string $s$ of length $j$ (for $1 \le j \le \ceil{\lg n} - 1$) such that $\seq(s) \neq \emptyset$.  Let $\hat s$ be $s$ with the last bit removed, and let $s'$ be $s$ with the last bit flipped.  Squad $S_s$ becomes nonempty when the members of squad $S_{\hat s}$ finish phase $j-1$ (or, in the case $j=1$, at the beginning of the game).  At this point, $S_{\hat s}$ splits into $S_s$ and $S_{s'}$.  Thus, at the beginning of phase $j$, each member of $S_s$ agrees with the robber in all but at most $|\seq(s')|$ active coordinates.  Moreover, for $i \in \seq(s')$, each cop in $S_s$ must make no more than $\mathrm{rad}(T_i)$ moves in the direction of coordinate $i$ before he agrees with the robber in that coordinate.  Hence, the members of $S_s$ finish phase $j$ after at most $N$ turns on which the robber either sits still or moves in the direction of some coordinate in $\seq(s)$, where $N = \sum_{i \in \seq(s')} \mathrm{rad}(T_i)$.

At any point in the game, the inactive coordinates (across all nonempty squads) partition $\{0, 1, \ldots, n-1\}$: this is clear at the beginning of the game, and is maintained under the splitting of squads.  Hence, on each robber move, at least one cop makes progress toward finishing either Step 1 or Step 2.  The number of rounds in which some cop makes progress in Step 1 is at most $\sum_{s} \sum_{i \in \seq(s')} \mathrm{rad}(T_i)$, where the sum is over all binary strings $s$ with length at most $\ceil{\lg n}-1$. However, this upper bound can be simplified as follows:
\begin{align*}
\sum_s \sum_{i \in \seq(s')} \mathrm{rad}(T_i) &= \sum_{j=1}^{\ceil{\lg n}-1} 
\sum_{\tiny \begin{array}[b]{c}
      s \text{ has}\\
      \text{length } j
    \end{array}}
\sum_{i \in \seq(s')} \mathrm{rad}(T_i)\\
                                               &= \sum_{j=1}^{\ceil{\lg n}-1} \sum_{i=0}^{n-1} \mathrm{rad}(T_i) = \left (\sum_{i=0}^{n-1} \mathrm{rad}(T_i) \right )(\ceil{\lg n}-1).
\end{align*}

Once cop $C_i$ has started Step 2, for $1 \le i \le k-2$, the robber can move in the direction of coordinates $2i$ and $2i+1$ no more than $\mathrm{rad}(T_{2i}) +\mathrm{rad}(T_{2i+1})-1$ times before $C_i$ captures him (by Lemma~\ref{lem:prodone}).  Likewise, if $n$ is odd, then once cop $C_{k-1}$ has entered Step 2, the robber can spend no more than $\mathrm{rad}(T_{n-1})-1$ turns either standing still or moving in the direction of coordinate $n-1$; if $n$ is even, then once cops $C_{k-2}$ and $C_{k-1}$ have entered Step 2, the robber can spend no more than $\floor{\mathrm{diam}(T_{n-2} \cart T_{n-1})/2}$ turns either standing still or moving in the direction of coordinates $n-2$ and $n-1$.  (This follows from the strategy given by Mehrabian~\cite{meh} for the 2-capture time of Cartesian products of two trees, with the caveat that the cops here require one extra turn to reach the initial positions used in Mehrabian's strategy.)  In either case, the number of rounds during which some cop makes progress in Step 2 is bounded above by $\sum_{i=0}^{n-1} \mathrm{rad}(T_i) - \floor{(n-1)/2} + 1$.  Thus, we have that
\begin{align*}
\capt(T_0 \cart T_1 \cart \cdots \cart T_{n-1}) &\le \left (\sum_{i=0}^{n-1} \mathrm{rad}(T_i) \right )(\ceil{\lg n}-1) + \sum_{i=0}^{n-1} \mathrm{rad}(T_i) - \floor{(n-1)/2} + 1\\
                                                  &= \left ( \sum_{i=0}^{n-1} \mathrm{rad}(T_i) \right )\ceil{\lg n} - \floor{\frac{n-1}{2}} + 1,
\end{align*}
as claimed.
\end{proof}

When $T_0 = T_1 = \cdots = T_{n-1} = K_2$, we have the following special case of Theorem~\ref{thm:upper}.

\begin{cor}\label{cor:upper_cube}
For $n$ an integer, we have that
$$
\capt(Q_n) \le n \ceil{\lg n} - \floor{\frac{n-1}{2}} + 1  = (1+o(1)) n \lg n .
$$
\end{cor}

Before concluding this section, we remark that the ideas in the proof of Theorem~\ref{thm:upper} can be applied in more generality.  Instead of playing the game on $T_0 \cart T_1 \cart \cdots \cart T_{n-1}$ for trees $T_i$, suppose we instead play on $G_0 \cart G_1 \cart \cdots \cart G_{n-1}$ for arbitrary graphs $G_i$.  In the proof of Theorem~\ref{thm:upper}, we had (essentially) $n/2$ cops, each of which sought to agree with the robber on all but two coordinates. In the game on $G_0 \cart G_1 \cart \cdots \cart G_{n-1}$, we can use $n$ teams of cops, with each team seeking to agree with the robber on all but one coordinate.  Suppose the team assigned to coordinate $i$ contains $k_i$ cops, where $k_i \ge c(G_i)$, for $0 \le i \le n-1$.  Once this team finishes Step 1, whenever the robber moves in the direction of coordinate $i$ (or stands still), the cops can change that coordinate according to a strategy realizing $\capt_{k_i}(G_i)$.  Eventually, some team captures the robber.  These modifications to the argument of Theorem~\ref{thm:upper} yield the following:

\begin{theorem}
For graphs $G_0, G_1, \ldots, G_{n-1}$ and integers $k_0, k_1, \ldots, k_{n-1}$ with $k_i \ge c(G_i)$, we have
$$\capt_K(G_0 \cart G_1 \cart \cdots \cart G_{n-1}) \le \left ( \sum_{i=0}^{n-1} \mathrm{rad}(G_i) \right )\ceil{\lg n} + \sum_{i=0}^{n-1} (\capt_{k_i}(G_i) - 1) + 1,$$
where $K = k_0 + k_1+ \cdots + k_{n-1}$.
\end{theorem}

Note that when each $G_i$ is a tree, this result is weaker than Theorem~\ref{thm:upper} both in terms of the bound on the capture time and in terms of the number of cops used.

\section{Lower bound}\label{lower}

We next show that $\capt(Q_n) \ge (1-o(1))\frac{1}{2}n \ln n$.  Our proof of this bound involves two strategies, one for the cops and one for the robber.  On each turn, a robber using the {\em random strategy} moves to an adjacent vertex chosen uniformly at random.  A cop using the {\em greedy strategy} moves one step closer to the robber on each turn; when there are multiple ways to do this, the cop chooses between them arbitrarily.  (For both strategies, the players may choose their starting locations ``intelligently''.)  As it turns out, the greedy strategy is in some sense the optimal response to the random strategy; we formalize this intuition below. The probability of an event $A$ is denoted by $\prob{A},$ and the expected value of a random variable $X$ is denoted by $\expect{X}$.

\begin{lemma}\label{greedy_optimal}
Fix $T$ a positive integer.  Consider a cop C attempting to capture a robber R using the random strategy on the cube $Q_n$, and suppose R has the next move.  For $d \le n$, define
$$p_d = \max \prob{\text{C catches R in under }T \text{ rounds} \given \text{C and R are distance }d \text{ apart}},$$
where the maximum is taken over all possible strategies for C.  Then the following inequalities hold for $1 \le k \le n/2$:
\begin{enumerate}[(i)]
\item $p_{2k-2} \ge p_{2k}$
\item $p_{2k} \ge p_{2k-1}$
\end{enumerate}
\end{lemma}
\begin{proof}
We prove both claims simultaneously by induction on $T$.  The claims are easily verified when $T = 1$, so fix $T > 1$.  For $d \le n$, define
$$p'_d = \max \prob{\text{C catches R in under }T-1 \text{ rounds} \given \text{C and R are distance }d \text{ apart}},$$
where the maximum is again taken over all strategies for C.  Note that always $p'_d \le p_d$.

For the proof of item (i), by the induction hypothesis, $p'_{2k-2} \ge p'_{2k} \ge p'_{2k-1}$ and $p'_{2k} \ge p'_{2k+2} \ge p'_{2k+1}$.  Suppose C is at distance $2k$ from R.  Let $d$ be the distance between C and R after one additional round of the game; since the distance between the players cannot change by more than 2 in a single round, $2k-2 \le d \le 2k+2$.  Thus, $p_{2k} \le p'_{2k-2} \le p_{2k-2}.$

For item (ii), suppose C and R begin distance $2k$ apart.  If C plays greedily, then after one additional round, the distance between the players is $2k-2$ with probability $2k/n$ and $2k$ with probability $(n-2k)/n$.  Hence, $p_{2k} \ge \frac{2k}{n}p'_{2k-2} + \frac{n-2k}{n}p'_{2k}$.

Suppose instead that C and R begin distance $2k-1$ apart.  If R moves closer to C, then after C's subsequent move, the distance between the players is one of $2k-3, 2k-2,$ and $2k-1$.  If instead R moves farther from C, then after C's move, the distance is one of $2k-1, 2k,$ and $2k+1$.  By the induction hypothesis, $p'_{2k-2} \ge p'_{2k-3}$ and $p'_{2k-2} \ge p'_{2k} \ge p'_{2k-1}$, so $p'_{2k-2} = \max \{p'_{2k-1},p'_{2k-2},p'_{2k-3}\}$.  Similarly, $p'_{2k} = \max \{p'_{2k-1}, p'_{2k}, p'_{2k+1}\}$.  Therefore,
$$p_{2k-1} \le \frac{2k-1}{n}p'_{2k-2} + \frac{n-2k+1}{n}p'_{2k} \le \frac{2k}{n}p'_{2k-2} + \frac{n-2k}{n}p'_{2k} \le p_{2k},$$
which completes the proof.
\end{proof}

Lemma~\ref{greedy_optimal} shows that, against a random robber, the cop should keep the distance between the players even and, subject to that, minimize the distance.  We have, therefore, that the following cop strategy is optimal against a random robber: if the robber is an even distance away, sit still; otherwise, play greedily.  The lemma shows that this strategy is optimal in that it maximizes the probability of capturing the robber ``quickly'', but in fact this optimality can be proved in a stronger sense.  A more detailed coupling argument shows that for any cop strategy, there is a greedy strategy that captures the robber no less quickly, assuming identical sequences of robber moves.  Since we need only the weaker property established in the lemma, we omit the proof of this stronger claim.

A single greedy cop will eventually capture a random robber.  In this scenario, the capturing process behaves similarly to the well-known ``coupon-collector'' process.  In the coupon-collector problem, there are $m$ types of coupons, and the aim is to collect at least one coupon of each type.  One coupon is obtained in each round, and the type of coupon is chosen uniformly at random.  It is well-known that the expected number of rounds needed to collect all $m$ types of coupon is $(1 + o(1)) m \ln m$.

When all but $i$ coupons have already been collected, the probability of obtaining a new coupon on the next turn is $i/m$, so the number of rounds needed to collect a new coupon is geometrically distributed with probability of success $i/m$.  As we will show, similar behaviour arises when a greedy cop plays against a random robber on $Q_n$.  In our analysis of $\capt(Q_n)$, we will need to bound the probability that the actual capturing time is significantly less than its expectation.  For this we use a slight generalization of a result by Doerr (see Theorem~1.24 in~\cite{Doe11}).  The proof of this generalization is quite similar to that of Doerr's original result, but we include the full proof for completeness.

\begin{lemma}\label{coupon}
Consider the coupon-collector process with $m$ coupons in total, of which all but $m_0$ have already been collected.  Let $X$ be a random variable denoting the number of rounds needed to collect the remaining $m_0$ coupons.  For $\eps > 0$, we have that
$$\prob{X < (1-\eps)(m-1) \ln m} \le \exp(-m^{-1+\eps}m_0).$$
\end{lemma}
\begin{proof}
Let $T = (1-\eps)(m-1)\ln m$.  For the $i$th uncollected coupon, where $i \in [m_0]$, let $X_i$ be an indicator random variable for the event that this type of coupon is collected within $T$ rounds.  It is straightforward to see that
\begin{eqnarray*}
\prob{X_i = 1} &=& 1 - \left (1 - \frac{1}{m}\right )^T \le 1 - \exp(-T/(m-1)) \\
&=& 1 - \exp(-(1-\eps)\ln m) = 1 - m^{-1+\eps},
\end{eqnarray*}
where the inequality above uses the standard inequality $(1 - 1/m) \ge \exp(-1/(m-1))$, which is valid for $m > 1$.
Fix $I \subseteq [m_0]$ and $j \in [m_0] \setminus I$.  Now
\begin{eqnarray*}
\prob{X_i = 1 \text{ for all } i \in I} &=&\prob{X_i = 1 \text{ for all } i \in I \given X_j = 1} \cdot \prob{X_j = 1}\\
                                          && + \prob{X_i = 1 \text{ for all } i \in I \given X_j = 0} \cdot \prob{X_j = 0}.
\end{eqnarray*}
But since
$$\prob{X_i = 1 \text{ for all } i \in I \given X_j = 0} \ge \prob{X_i = 1 \text{ for all } i \in I},$$
it follows that
$$\prob{X_i = 1 \text{ for all } i \in I \given X_j = 1} \le \prob{X_i = 1 \text{ for all } i \in I}.$$
Thus,
$$\prob{X_i = 1 \text{ for all } i \in I \cup \{j\}} \le \prob{X_i = 1 \text{ for all } i \in I} \cdot \prob{X_j = 1}.$$
Now induction yields
$$\prob{X_i = 1 \text{ for all } i \in [m_0]} \le \prod_{i=1}^{m_0} \prob{X_i = 1} \le (1 - m^{-1+\eps})^{m_0} \le \exp(-m^{-1+\eps}m_0),$$
where the last inequality follows because $(1+x) \le \exp(x)$ for all $x$.
\end{proof}

We are now ready to prove our lower bound on $\capt(Q_n)$, and the proof of this result concludes the proof of Theorem~\ref{time}.  We state the lower bound in greater generality.

\begin{theorem}\label{main2}
Fix a positive constant $d$.  On $Q_n$, a robber can escape capture against $n^d$ cops for at least $(1-o(1))\frac{1}{2} n \ln n$ rounds.
\end{theorem}
\begin{proof}
Let $T = \frac{1}{2} (n-1) \ln n$ and let
$$
\eps = \eps(n) = \frac {\ln ((4d+1) \ln n)}{\ln n} = \frac {O(\ln \ln n)}{\ln n} = o(1).
$$
Fix a cop strategy.  We claim that for sufficiently large $n$, a random robber evades capture for at least $(1-\eps)T$ rounds with positive probability; it would then follow that some particular set of random choices, and hence some particular deterministic strategy, allows the robber to survive this long.

For the initial placement, the robber aims to choose a starting vertex that is relatively far from \emph{all} cops.  The number of vertices of $Q_n$ within distance $k$
of any given cop is $\sum_{i=0}^k \binom{n}{i}$, so the number of vertices within distance $k$ of any of the $n^d$ cops is at most
$$
n^d\sum_{i=0}^k \binom{n}{i} \le n^d (k+1) \binom{n}{k} \le 2 n^d k \left( \frac {ne}{k} \right)^k \le n^{d+1} \left( (4e)^{1/4} \right)^n < n^{d+1} 1.85^n,
$$
when $k = n/4$. For sufficiently large $n$ we have $n^d\sum_{i=0}^k \binom{n}{i} < 2^n$ , so the robber can choose an initial position at least distance $n/4+1$ from every cop.  (In fact more careful analysis permits $k = (1 - o(1))n/2$, but using $n/4+1$ simplifies later computations without affecting the asymptotics of the final result.)

To analyze the effectiveness of the robber's strategy, we use the following approach.  First, we focus on a single cop.  We determine the probability that this cop captures the robber within $(1-\eps)T$ rounds.  We then apply the union bound to obtain an upper bound on the probability that any of the cops can capture the robber within $(1-\eps)T$ rounds.  If this probability is less than 1, then with positive probability the robber evades capture.

Suppose cop C is at distance $2k$ from the robber (on the robber's turn).  By Lemma~\ref{greedy_optimal}, we may assume C uses a greedy strategy.  On any given round, the robber moves toward C with probability $2k/n$.  Thus, after the robber's move and C's response, the distance between the two is $2k-2$ with probability $2k/n$ and $2k$, otherwise.  Letting $X_i$ denote the number of rounds for which the cop remains at distance $2i$ from the robber, we see that $X_i$ is geometrically distributed with probability of success $2i/n$; the number of rounds needed for C to catch the robber is at most $\sum_{i=1}^{k} X_i$.  Hence, this process is equivalent to the coupon-collector process with $n/2$ coupons, if we suppose that all but $k$ coupons have already been collected.  Since the robber begins at least distance $n/4+1$ from C, we may (again by Lemma~\ref{greedy_optimal}) take $k = \ceil{n/8}$ to obtain a lower bound on the expected length of the process:
$$\sum_{i=1}^{\ceil{n/8}} \expect{X_i} = \sum_{i=1}^{\ceil{n/8}} \frac{n}{2i} = \frac{n}{2} \sum_{i=1}^{\ceil{n/8}} \frac{1}{i},$$
which tends to $\frac{n}{2}(\ln n - \ln 8 + \gamma)$, where $\gamma \approx 0.557$ denotes the Euler-Mascheroni constant.

By Lemma~\ref{coupon}, the probability that C captures the robber in under $(1-\eps)T$ rounds is at most $\exp(-(n/2)^{-1+\eps}n/8)$, which is
asymptotically $\exp(-(n/2)^{\eps}/4)$.  By the union bound, the probability that any of the cops captures the robber in under $(1-\eps)T$ rounds is thus, at most $(1 + o(1))n^d\exp(-(n/2)^{\eps}/4)$. Since
$$
\frac {(n/2)^{\eps}}{4} = \frac {1}{4} \exp \left( \eps \Big(\ln n + O(1) \Big) \right) = (1+o(1)) \frac 14 (4d+1) \ln n > \left( d + \frac {1}{5} \right) \ln n,
$$
this probability tends to 0 as $n$ tends to infinity.  Hence, with probability approaching 1 as $n \to \infty$, the robber escapes capture for at least $(1-\eps)T$ rounds.
\end{proof}

It remains open to compute the exact value of the capture time of the hypercube $Q_n$, although Theorem~\ref{time} derives it up to a constant factor. We conjecture that $\mathrm{capt}(Q_n) = (1 + o(1))\frac{1}{2}n \ln n.$

\end{document}